\documentclass[A4paper,11pt,english]{article}
\usepackage{babel}
\usepackage{graphicx}
\usepackage{amsmath}
\usepackage[utf8]{inputenc}
\usepackage{amsfonts}
\usepackage{amssymb}
\usepackage{amsmath}
\usepackage{amsthm}
\usepackage{xcolor}
\usepackage{array}
\usepackage{cite}
\usepackage{hyperref}
\title{\bf Saturation and Localization Results for Max-product Generalized Sampling Operators based on Centered Bell-shaped Kernels}
\author{ {\bf Lorenzo Boccali$^{1,2}$} \hskip1cm {\bf Gianluca Vinti$^{1}$} \\ 
	$^{1}$Department of Mathematics and Computer Science \\
	University of Perugia\\
	1, Via Vanvitelli, 06123 Perugia, Italy \\ 
	$^{2}$Department of Mathematics and Computer Science ``Ulisse Dini" \\ University of Florence \\ 67/a, Viale Morgagni, 50134 Florence, Italy \\
	{\small {\tt lorenzo.boccali@unifi.it}} -  {\small {\tt gianluca.vinti@unipg.it}} }
\date{}

\newtheorem{definizione}{Definition}[section]
\newtheorem{lemma}{Lemma}[section]
\newtheorem{teorema}{Theorem}[section]
\newtheorem{cor}{Corollary}[section]
\theoremstyle{definition}
\newtheorem{remark}{Remark}[section]

\newcommand{\R}{\mathbb{R}}

\newcommand{\assolutol}{\lvert}
\newcommand{\assolutor}{\rvert}

\begin{document}
	\maketitle
	\begin{abstract}
	In this paper, we establish the saturation order and a local inverse result for the uniform approximation of non-negative, bounded, and uniformly continuous functions on $\mathbb{R}$ by max-product generalized sampling operators based on suitable kernel functions. In particular, assuming that the kernel is an even centered bell-shaped function, we first show that $1/w$, $w>0$, is the uniform saturation order, with the corresponding saturation class coinciding with the class of non-negative constant functions. This means that $1/w$ is the best possible rate of convergence that the max-product generalized sampling operators can achieve when approximating non-trivial (i.e., non-constant) non-negative, bounded, and uniformly continuous functions on $\mathbb{R}$. Moreover, it is known that, for Lipschitz continuous functions on $\mathbb{R}$, the approximation order is $1/w$ as $w \to +\infty$. Here, we show that this result can be locally reversed. Specifically, we prove that if $f$ can be approximated at the rate $1/w$ on a compact interval $[a,b]\subset\mathbb{R}$, then $f$ is Lipschitz continuous on $[a,c]$ for every $c \in [a,b)$ whenever $0<a<b$, and on $[c,b]$ for every $c \in (a,b]$ whenever $ a<b<0$. Finally, under the same assumptions on the kernel, we establish a strong localization result for sequences of truncated max-product generalized sampling operators in the case of strictly positive and bounded functions defined on $[0,1]$.  All these results extend previous results of Coroianu and Gal, which were established only for specific sinc-type kernels, to a broader class of kernel functions.   
	\end{abstract}

\medskip\noindent
{\small {\bf AMS subject classification:} 41A25, 41A05   \newline
	{\small {\bf Key Words:} Centered bell-shaped functions, max-product generalized sampling operators, saturation order, local inverse result, localization result

\section{Introduction}
Within the framework of Approximation Theory, determining the saturation order and establishing inverse results for a given approximation process are among the most interesting and challenging problems (see, e.g., \cite{May1976,Ivanov1983,CoroianuGal2012,CostarelliVinti2019,CostarelliVinti2019ter,Cantarini2020,CantariniCoroianu2021}). \newline In general, investigating the saturation order of a family of operators $(L_{w})_{w>0}$ amounts to determining a class of functions $\mathcal{D}$, a certain subclass $\mathcal{E} \subset \mathcal{D}$ consisting of trivial functions, and a positive non-increasing function $\varphi(w)$, $w>0$, such that there exists a function $g \in \mathcal{D} \setminus \mathcal{E}$ with $\|L_{w}(g)-g\|=\mathcal{O}(\varphi(w))$ as $w\rightarrow+\infty$, and with the property that:
\begin{description}
\item[(I)] for any $f \in \mathcal{D}$ and $\| L_{w}(f)-f \|=o(\varphi(w))$ as $w \rightarrow +\infty$, then $f \in \mathcal{E}$. 
\end{description}
Here, $\| \cdot \|$ denotes any norm defined on $\mathcal{D}$. The function $\varphi(w)$ is called the \textit{saturation order} of the approximation process $(L_{w})_{w>0}$; it represents the best possible rate of approximation that the operators $L_{w}$ can achieve when approximating functions belonging to $\mathcal{D}$.  \newline On the other hand, it is also of interest to establish, whenever possible, an inverse result for $L_{w}$, which consists in proving that 
\begin{description}
\item[(II)] if $\| L_{w}(f)-f \| = \mathcal{O}(\varphi(w))$ as $w \rightarrow + \infty$, then $f$ belongs to a non-trivial subclass of $\mathcal{D}$. 
\end{description} 
In the present paper, we face these problems for the following family of max-product generalized sampling operators: 
\begin{equation*}
S_{w}^{\chi}(f)(x):=\frac{\displaystyle \bigvee_{k \in \mathbb{Z}}f\left(\frac{k}{w}\right)\chi(wx-k)}{\displaystyle \bigvee_{k \in \mathbb{Z}} \chi(wx-k)}, \quad x \in \mathbb{R}, \ w>0, \quad \quad \quad \quad (\diamondsuit)
\end{equation*} 
where $f:\mathbb{R} \rightarrow \mathbb{R}$ is any bounded function, $\chi:\mathbb{R} \rightarrow \mathbb{R}$ is a general kernel function satisfying suitable assumptions, and the symbol $\bigvee$ (as is common in the literature, see, e.g., \cite{Anastassiou2018,Gungor2018,Holhos2018,Bajpeyi2024,Bede2009}) denotes the supremum operator, i.e., the supremum of the terms $f\left(\frac{k}{w}\right)\chi(wx-k)$ and $\chi(wx-k)$, respectively, for $k \in \mathbb{Z}$, $w>0$, and $x \in \mathbb{R}$. \newline  The non-linear (more precisely, sub-linear) operators in ($\diamondsuit$), originally introduced as a sequence of operators by Coroianu et al. \cite{CoroianuCostarelli2019} for the uniform approximation of non-negative, bounded, and uniformly continuous functions on $\mathbb{R}$, represent the max-product counterparts of the well-known generalized sampling series introduced by the German mathematician P.L. Butzer (see, e.g., \cite{Butzer1987,ButzerStens1993,Bardaro2010,Riesz1984}), obtained by replacing the series with the symbol $\bigvee$. This process of non-linear generalization, applied by Bede, Coroianu, and Gal to several families of linear approximation operators in their  comprehensive monograph \cite{Bede2016}, offers a twofold advantage in the case of generalized sampling operators. On the one hand, max-product generalized sampling operators provide a faster order of uniform approximation than their linear counterparts. On the other hand, they can be constructed using a broader class of kernel functions than classical (linear) generalized sampling operators. Indeed, in the max-product setting, convergence results can be obtained without requiring the kernel $\chi$ to satisfy the classical assumptions of discrete approximate identities, which are instead essential in Butzer's theory (see, e.g., \cite{Bardaro2010,Butzer1987,Bardaro2013}). \newline Based on the above motivations, sequences of max-product generalized sampling operators have been extensively investigated in the classical continuous setting, leading to pointwise and uniform convergence results, together with a Jackson-type estimate in terms of the uniform modulus of continuity in the case of non-negative, bounded, and uniformly continuous functions on $\mathbb{R}$.
\newline However, the study of the saturation order and inverse results with respect to the uniform norm remains an open problem. A key step towards successfully addressing these questions lies in the capability to express the denominator on the right-hand side of $(\diamondsuit)$ in a suitable form (see Lemma \ref{lemma1}). As one might expect, this property strictly depends on the choice of the kernel function $\chi$. Unfortunately, not all kernels associated with the operators $S_{w}^{\chi}$ possess this property. Therefore, we are compelled to restrict our analysis to a specific class of kernel functions. More precisely, we assume that the function $\chi$ appearing in the definition of ($\diamondsuit$) is an even centered bell-shaped function, according to the notion introduced by Cardaliaguet and Euvrard in \cite{Cardaliaguet1992}.  \newline In this paper, under suitable assumptions on the kernel $\chi$, we are able to solve problem \textbf{(I)} for the above operators with $\mathcal{D}$ coinciding with the space of non-negative, bounded, and uniformly continuous functions defined on the whole real axis and endowed with the usual sup-norm, $\varphi(w)=\frac{1}{w}$, $w>0$, and with the trivial class $\mathcal{E}$ consisting of non-negative constant functions. \newline However, for what concerns question \textbf{(II)}, we are only able to establish a local inverse approximation result. More precisely, if $f:\mathbb{R}\to \mathbb{R}_{0}^{+}$ is a bounded and uniformly continuous function such that: 
\begin{equation*}
\sup_{x \in [a,b]} \lvert S_{w}^{\chi}(f)(x)-f(x) \rvert = \mathcal{O}\left(\frac{1}{w}\right) \quad \text{as} \ w \rightarrow +\infty,
\end{equation*}
then $f$ is Lipschitz continuous on $[a,c]$ for every $c \in [a,b)$ whenever $0<a<b$, and on $[c,b]$ for every $c \in (a,b]$ whenever $a<b<0$. This result turns out to be consistent with the direct approximation theorem concerning the problem of the rate of convergence, already solved for sequences of max-product generalized sampling operators in \cite{CoroianuCostarelli2019}. Indeed, it has been proved that the order of approximation in the space of (non-negative) Lipschitz functions on $\mathbb{R}$ is $\frac{1}{n}$, as $n \rightarrow +\infty$.
\newline In the final section, in order to provide a comprehensive characterization of the approximation properties of the max-product generalized sampling operators, we turn our attention to localization results. In addition to offering a higher order of approximation than their linear counterparts, it is well-known that the max-product versions of several families of linear discrete operators generally exhibit strong localization properties (see, e.g., \cite{CoroianuGal2013,CoroianuGal2013bis,CoroianuGal2014bis}). More precisely, these non-linear operators provide a much better, possibly optimal, local approximation of the target function. Specifically, if two functions $f$ and $g$ coincide on a strict subinterval $I$ of their domain, then, for every subinterval $I'$ strictly contained in $I$, the max-product operators applied to $f$ and $g$, denoted by $L_{n}(f)$ and $L_{n}(g)$, respectively, coincide on $I'$ for all sufficiently large $n\in \mathbb{N}$. Within the framework of sampling operators, in \cite{CoroianuGal2012}, the authors established a strong localization result for sequences of truncated max-product sampling operators based solely on sinc-type kernels, such as the well-known sinc/Whittaker and Fejér kernels. Here, we show that, by still working with even centered bell-shaped kernel functions, it is possible to recover the same result for the following family of truncated max-product generalized sampling operators:
\begin{equation*}
S_{n}^{\chi}(f)(x)=\frac{\displaystyle\bigvee_{k=0}^{n}f\left(\frac{k}{n}\right)\chi(nx-k)}{\displaystyle \bigvee_{k=0}^{n}\chi(nx-k)}, \quad x \in [0,1], \ n \in \mathbb{N},
\end{equation*}
where $f:[0,1]\rightarrow \mathbb{R}^{+}$ is a strictly positive bounded function. In particular, we prove that if $f$ and $g$ are bounded functions with strictly positive lower bounds and coincide on a given subinterval $[a,b] \subset [0,1]$, then, for sufficiently large $n \in \mathbb{N}$,  the operators $S_{n}^{\chi}(f)$ and $S_{n}^{\chi}(g)$ coincide on every subinterval $[c,d]$ sufficiently close to $[a,b]$. As a direct consequence, if $f$ is a strictly positive function that is constant on an interval $[a,b] \subset [0,1]$, then, for sufficiently large $n \in \mathbb{N}$, $S_{n}^{\chi}(f)$ attains the same constant value on subintervals sufficiently close to $[a,b]$. This shows that the truncated max-product generalized sampling operators based on suitable kernels are particularly effective for the local reconstruction of strictly positive, non-smooth, but continuous functions that are constant on proper subintervals of $[0,1]$. This property may lead to interesting applications in signal processing and represents a clear advantage of the max-product version of generalized sampling operators over their linear counterparts, which do not exhibit this useful property. \newline The paper is organized as follows. In Section \ref{Sezione2}, we recall some notation and preliminary results concerning the max-product generalized sampling operators and present examples of specific kernels for which the theory developed in this paper applies. Section \ref{Sezione3} is devoted to the study of the uniform saturation order of the operators $S_{w}^{\chi}$ in ($\diamondsuit$), while in Section \ref{Sezione4}, a local inverse result is established.  Finally, in Section \ref{Sezione5}, we prove a strong localization result for the truncated version of the considered operators. 
\section{Preliminaries}
\label{Sezione2}
In this paper, we use the notation $f:I \rightarrow \mathbb{R}$ to denote real-valued functions defined on the compact interval $I=[a,b] \subset \mathbb{R}$ or $I=\mathbb{R}$. In accordance with this notation, let $C(I)$ denote the space of all continuous functions $f:[a,b] \rightarrow \mathbb{R}$, if $I=[a,b]$, and the space of all bounded and uniformly continuous functions $f:\mathbb{R}\to\mathbb{R}$, when $I=\mathbb{R}$. Moreover, by $C_{+}(I)$ we denote the subspace of $C(I)$ consisting of non-negative valued functions. On the space $C(I)$ we consider the usual sup-norm $\| \cdot \|_{\infty}$. Furthermore, by $\textnormal{Lip}(I)$ we denote the subspace of $C(I)$  consisting of Lipschitz continuous functions on $I$, i.e., the space of all functions for which there exists a positive constant $L$ such that
\begin{equation*}
	\lvert f(x)-f(y) \rvert \le L \lvert x-y \rvert,
\end{equation*} 
for every $x$, $y \in I$. 
\newline Consistently with a wide literature (see, e.g., \cite{Coroianu2010,Coroianu2011bis,Wu2025,Ellidokuz2025,Angamuthu2024}), we denote by $\bigvee$ the max symbol, defined as follows.
\begin{definizione}
	Let $(A_{k})_{k \in \mathbb{Z}}$ be a family of real numbers. We denote:
	\begin{equation}
	\label{equazione1}
		\bigvee_{k \in \mathcal{J}}A_{k}:=\sup\bigl\{A_{k}: k \in \mathcal{J}\bigr\},
	\end{equation}
	for any set of indices $\mathcal{J} \subseteq \mathbb{Z}$. Note that if $\mathcal{J}$ is finite, then the supremum on the right-hand side of (\ref{equazione1}) reduces to a maximum.
\end{definizione}
The supremum (or the maximum in the finite case) operation plays a crucial role in the definition of the non-linear operators studied in this paper. More precisely, we will consider max-product generalized sampling operators based on general kernel functions, namely bounded and measurable functions $\chi:\mathbb{R} \rightarrow \mathbb{R}$ satisfying the following suitable assumptions:
\begin{description}
\item[{($\chi$1)}] For a suitable $\beta>0$, there holds:
\begin{equation*}
	m_{\beta}(\chi):=\sup_{u \in \mathbb{R}}\bigvee_{k \in \mathbb{Z}}\lvert \chi(u-k) \rvert \cdot \lvert u-k \rvert^{\beta}=\sup_{y \in \mathbb{R}} \lvert\chi(y) \rvert \cdot \lvert y \rvert^{\beta}<+\infty,
\end{equation*}
i.e., the generalized absolute moment of order $\beta$ of $\chi$ is finite;
\item[($\chi$2)] For a suitable constant $a_{\chi}>0$, we have
\begin{equation*}
	\inf_{x \in \left[-\frac{1}{2},\frac{1}{2}\right]} \chi(x)=:a_{\chi}>0, \quad \textnormal{if} \ I=\mathbb{R},
\end{equation*}
or 
\begin{equation*}
	\inf_{x \in [-1,1]} \chi(x)=:a_{\chi}>0, \quad \textnormal{if} \ I=[a,b]\subset\mathbb{R};
\end{equation*}
\item[($\chi$3)] $\chi$ is continuous on $\mathbb{R}$, $\chi(x)\ge0$ for every $x \in \mathbb{R}$, and moreover $\lim_{\lvert x \rvert \rightarrow +\infty}\chi(x)=0$;
\item[($\chi$4)] $\chi$ is an even function;
\item[($\chi$5)] $\chi(x)$ is non-decreasing for $x <0$ and non-increasing for $x \ge 0$.
\end{description}
\begin{remark}
$(a)$ Note that, from conditions $(\chi4)$ and $(\chi5)$, it is easy to deduce that $\chi\left(\frac{1}{2}\right)=a_{\chi}>0$ if $I=\mathbb{R}$, and $\chi(1)=a_{\chi}>0$, when $I=[a,b]$, where $a_{\chi}$ is the constant defined in condition $(\chi2)$. 
\newline $(b)$ The function $\chi(x)$ is an even ``centered bell-shaped function", according to the definition given by Cardaliaguet and Euvrard in \cite{Cardaliaguet1992}.
\end{remark}
From now on, unless otherwise stated, we assume that the kernel function $\chi$ satisfies conditions $(\chi1)-(\chi5)$. We first recall the following well-known results concerning the kernel $\chi$.
\begin{lemma}\textnormal{(see Lemma 2.4 of \cite{CostarelliVinti2016})}
	If $\chi:\R \rightarrow \R$ is bounded and such that $\chi(x)=\mathcal{O}(\lvert x \rvert^{-\alpha})$ as $\assolutol x \assolutor \rightarrow +\infty$, for $\alpha >0$, then:
	\begin{equation*}
		m_{\beta}(\chi)<+\infty, \ for \ every \ \ 0 \le \beta \le \alpha.
	\end{equation*}
\end{lemma}
\begin{lemma}\textnormal{(see Lemma 2.2 of \cite{CoroianuCostarelli2019})}
	\label{lemma2.1}
	Let $\chi:\mathbb{R} \rightarrow \mathbb{R}$ be a bounded function satisfying $(\chi1)$ with $\beta>0$. Then, there holds:
	\begin{equation*}
		m_{\upsilon}(\chi)<+\infty, \quad for \ every \ \ 0 \le \upsilon \le \beta.
	\end{equation*} 
	In particular, $m_{0}(\chi)=\|\chi\|_{\infty}$. 
\end{lemma}  
\begin{lemma}\textnormal{(see Lemma 2.3 of \cite{CoroianuCostarelli2019})}
	\label{lemma2.2}
	Let $\chi:\mathbb{R}\rightarrow \mathbb{R}$ be a function satisfying $(\chi2)$ and let $[a,b]\subset\mathbb{R}$ be a compact interval. Then, the following inequalities:
	\begin{equation*}
	\hskip2.5cm	\bigvee_{k \in \mathbb{Z}} \chi(nx-k) \ge a_{\chi}>0, \quad x \in \mathbb{R},
	\end{equation*}
and
	\begin{equation*}
	\hskip3.0cm	\bigvee_{k=\lceil na \rceil}^{\lfloor nb \rfloor} \chi(nx-k) \ge a_{\chi}>0, \quad x \in [a,b],
	\end{equation*}
	hold for every $n \in \mathbb{N}$, where $a_{\chi}$ is the constant arising from condition $(\chi2)$, and, here and in what follows, $\lceil\cdot \rceil$ and $\lfloor \cdot \rfloor$ denote, respectively, the ``ceiling" and the ``integer part" of a real number.
\end{lemma} 
We are now in a position to provide the definition of the max-product generalized sampling operators.  
\begin{definizione}
\label{definizione2.2}
Let $f:I \rightarrow \mathbb{R}$ be a bounded function. The max-product generalized sampling operators based on the kernel $\chi$ are defined as follows:
\begin{equation}
\label{operatoriequazione2}
S_{n}^{\chi}(f)(x):=\frac{\displaystyle \bigvee_{k \in \mathcal{J}_{n}} f\left(\frac{k}{n}\right)\chi(nx-k)}{\displaystyle \bigvee_{k \in \mathcal{J}_{n}}\chi(nx-k)}, \quad x \in I,
\end{equation} 
where $\mathcal{J}_{n}=\mathbb{Z}$ if $I=\mathbb{R}$, or $\mathcal{J}_{n}=\{k \in \mathbb{Z}: k=\lceil na \rceil,\dots,\lfloor nb \rfloor\}$ when $I=[a,b]$, for each $n \in \mathbb{N}$.
\end{definizione}
Note that, by the inequalities recalled in Lemma \ref{lemma2.2}, the denominator on the right-hand side of (\ref{operatoriequazione2}) is strictly positive for every $x \in I$ and $n \in \mathbb{N}$. Moreover, by the boundedness of $f$ and Lemma \ref{lemma2.1}, it immediately follows that:
\begin{equation*}
\|S_{n}^{\chi}(f) \|_{\infty} \le \frac{m_{0}(\chi)}{a_{\chi}} \| f \|_{\infty}<+\infty.
\end{equation*}
Therefore, the operator $S_{n}^{\chi}$ is well-defined for every $n \in \mathbb{N}$. Since these operators are constructed using the supremum (or maximum) operation, they are non-linear; more precisely, they are sub-additive and positively homogeneous. \newline The theory of the max-product generalized sampling operators has been thoroughly investigated in \cite{CoroianuCostarelli2019}, within suitable spaces of continuous functions. In this framework, pointwise and uniform convergence has been established for non-negative continuous functions defined on compact intervals, as well as for non-negative, bounded, and uniformly continuous functions $f: \mathbb{R} \rightarrow \mathbb{R}^{+}_{0}$. In addition, a quantitative Jackson-type estimate with respect to the uniform norm has been obtained via the well-known modulus of continuity (see (\ref{modulodicontinuità})) of the function being approximated. This result shows that, as is often the case, max-product operators improve the order of approximation compared to their corresponding classical (linear) counterparts (see, e.g., \cite{Coroianu2011tris,Holhos2018bis,CoroianuCostarelli2021}). More recently, the study of the approximation properties of the non-linear operators $S_{n}^{\chi}$ has been extended to the usual $L^{p}$-spaces, $1 \le p< +\infty$. In particular, in \cite{BoccaliVinti0}, we derive quantitative estimates for the approximation error in the $L^{p}$-norm in the case of non-negative bounded functions defined on $[-1,1]$, where the upper bound is expressed in terms of the so-called $\tau$-modulus. It is worth recalling that, to prove the above results, it is not necessary to require assumptions $(\chi3)-(\chi5)$ on the kernel function; only conditions $(\chi1)-(\chi2)$ are needed. However, we will show that these additional assumptions play a crucial role in the proofs presented in the next sections.  \newline We conclude this section by providing some examples of kernels $\chi$ to which the theory developed in this paper can be applied. First of all, the Gaussian function and the hyperbolic secant (including their scaled versions), respectively defined by:
\begin{equation*}
\chi_{g}(x):=e^{-ax^{2}}, \ \ a>0, \quad  \chi_{h}(x):=\frac{2}{e^{x}+e^{-x}},
\end{equation*}
represent fundamental examples of (smooth) even centered bell-shaped functions. More generally, we recall a useful criterion for generating bell-shaped functions: it is well-known that the primitive of any bell-shaped function is a squashing function. Consequently, for example, the derivative of the logistic function and the derivative of the arctangent function (and their scaled versions), respectively given by:
\begin{equation*}
\chi_{l}(x):=\frac{e^{-x}}{(1+e^{-x})^{2}}, \quad \chi_{a}(x):=\frac{1}{1+x^{2}},
\end{equation*} 
can be used as kernels to obtain all the approximation results presented in this paper. 
Another kernel function that satisfies assumptions $(\chi1)-(\chi5)$ is the following:
\begin{equation*}
	\chi(x)=\begin{cases}
		\displaystyle \frac{2}{\pi}\arctan{\left(\frac{1}{x^{2}}\right)}, & x \ne 0, \\
		\displaystyle 1, & x=0.
	\end{cases}
\end{equation*}
Furthermore, we may also consider kernels with compact support, such as the so-called central B-spline of order $3$ (see, e.g., \cite{BardaroKarsli2011,BardaroButzer2007}) defined by:
\begin{equation}
	\label{centralb-spline}
	\chi_{b}(x):= \begin{cases} \displaystyle \frac{3}{4}-x^2, & \displaystyle  \assolutol x\assolutor \le \frac{1}{2},\\
		\displaystyle \frac{1}{2}\left(\frac{3}{2}-\assolutol x \assolutor\right)^2, & \displaystyle \frac{1}{2} < \assolutol x \assolutor \le \frac{3}{2},\\
		\displaystyle 	0, & \displaystyle \assolutol x \assolutor > \frac{3}{2}.
	\end{cases}
\end{equation}   
It is easy to verify that $\chi_{b}(x)$ is a centered bell-shaped function whose support is contained in the interval $\left[-\frac{3}{2},\frac{3}{2}\right]$. \newline Finally, other examples of even centered bell-shaped kernels can be obtained from the so-called sigmoidal functions, namely measurable functions $\sigma: \mathbb{R} \to \mathbb{R}$ such that
\begin{equation*}
\lim_{x \to -\infty}\sigma(x)=0 \ \textnormal{and} \ \lim_{x \to + \infty} \sigma(x)=1.
\end{equation*} 
More precisely, it is possible to prove (see, e.g., \cite{CostarelliVinti2016,CostarelliVinti2016ter}) that the following non-negative function
\begin{equation*}
\phi_{\sigma}(x):=\frac{1}{2}[\sigma(x+1)-\sigma(x-1)], \quad x \in \mathbb{R},
\end{equation*}
is an even centered bell-shaped function, provided that the non-decreasing sigmoidal function $\sigma$ satisfies the following conditions:
\begin{description}
	\item[$(\Sigma1)$] $\sigma(x)-\frac{1}{2}$ is an odd function;
	\item[$(\Sigma2)$] $\sigma \in C^{2}(\mathbb{R})$ is concave for $x \ge 0$;
	\item[$(\Sigma3)$] $\sigma(x)=\mathcal{O}(\lvert x \rvert^{-\alpha})$ as $x \rightarrow-\infty$, for some $\alpha>0$.
\end{description}
In particular, the so-called density function $\phi_{\sigma}$ satisfies condition $(\chi1)$ for every $0 \le \beta \le \alpha$ and condition $(\chi2)$ with $a_{\phi_{\sigma}}=\phi_{\sigma}\left(\frac{1}{2}\right)$ if $I=\mathbb{R}$ and $a_{\phi_{\sigma}}=\phi_{\sigma}(1)$ if $I=[a,b]$. \newline It follows that all the results established in the following sections can be applied to the so-called max-product neural network (NN) operators \cite{CostarelliVinti2016}. 
\section{The saturation order}
\label{Sezione3}
In this section, we investigate the saturation order with respect to uniform approximation for the max-product generalized sampling operators in the unbounded case $I=\mathbb{R}$. To this end, the integer parameter $n \in \mathbb{N}$ needs to be replaced by a real parameter $w>0$, which serves as the index of the operators. Therefore, according to Definition \ref{definizione2.2}, we consider the following family of non-linear operators:
\begin{equation}
	\label{equazione4saturazione}
	S_{w}^{\chi}(f)(x)=\frac{\displaystyle\bigvee_{k \in \mathbb{Z}}f\left(\frac{k}{w}\right)\chi(wx-k)}{\displaystyle \bigvee_{k \in \mathbb{Z}}\chi(wx-k)}, \quad x \in \mathbb{R}, \ w>0,
\end{equation}
where $f:\mathbb{R}\rightarrow \mathbb{R}$ is any bounded function and $\chi:\mathbb{R}\rightarrow\mathbb{R}$ is a kernel function satisfying conditions $(\chi1)-(\chi5)$. 
In \cite{CoroianuCostarelli2019}, the authors proved that, if $f \in C_{+}(\mathbb{R})$, the sequence $(S_{n}^{\chi}(f))_{n \in \mathbb{N}}$ converges uniformly to $f$ on $\mathbb{R}$ as $n \rightarrow +\infty$. It is easy to verify that the same uniform convergence result holds for the family of operators $(S_{w}^{\chi})_{w>0}$. Similarly, by Theorem 4.1 of \cite{CoroianuCostarelli2019}, it follows that, if condition $(\chi1)$ is satisfied with $\beta \ge 1$, then there exists a positive constant $M$, depending only on $\chi$, such that 
\begin{equation}
\label{Jackson-typeestimate}
\| S_{w}^{\chi}(f) - f \|_{\infty} \le M \omega\left(f,\frac{1}{w}\right), \ \ \text{as} \ w \rightarrow +\infty,
\end{equation}
where
\begin{equation}
\label{modulodicontinuità}
\omega\left(f,\delta\right):=\sup\left\{\lvert f(x)-f(y) \rvert: x, y \in \mathbb{R}, \lvert x-y \rvert \le \delta\right\}, \quad \delta >0,
\end{equation}
denotes the uniform modulus of continuity of the function $f \in C_{+}(\mathbb{R})$. The Jackson-type estimate in (\ref{Jackson-typeestimate}) implies that, if a non-negative function $f$ belongs to $\textnormal{Lip}(\mathbb{R})$, then: 
\begin{equation}
\label{ordineLip}
\| S_{w}^{\chi}(f)-f \|_{\infty}=\mathcal{O}\left(\frac{1}{w}\right), \ \ \text{as} \ w \rightarrow +\infty,
\end{equation}  
that is, the order of uniform approximation for the max-product generalized sampling operators in (\ref{equazione4saturazione}) in the space $\textnormal{Lip}(\mathbb{R})$ is $\frac{1}{w}$ as $w \rightarrow +\infty$. 
\newline In order to study the saturation order of the operators $S_{w}^{\chi}$ in (\ref{equazione4saturazione}) based on a kernel function $\chi$ satisfying assumptions $(\chi1)-(\chi5)$, we need the following fundamental preliminary result. The proof is analogous to that of Lemma 4.1 in \cite{BoccaliVinti0} and is mainly based on assumptions $(\chi4)$ and $(\chi5)$. 
\begin{lemma}
\label{lemma1}
 For any $j \in \mathbb{Z}$ and $w>0$, we have:
\begin{equation*}
	\begin{split}
		&\bigvee_{k \in \mathbb{Z}}\chi(wx-k)=\chi(wx-j), \quad for \ every \ x \in \left[\frac{j}{w},\frac{j}{w}+\frac{1}{2w}\right], \\
		&\bigvee_{k \in \mathbb{Z}}\chi(wx-k)=\chi(wx-(j+1)), \quad for \ every \ x\in \left[\frac{j}{w}+\frac{1}{2w},\frac{j}{w}+\frac{1}{w}\right].
	\end{split}
\end{equation*}
\end{lemma}
This lemma allows us to obtain an explicit representation of the denominator appearing in the definition of $S_{w}^{\chi}(f)(x)$, which is crucial for proving the main result of this section. Namely,
\begin{teorema}
\label{teoremadisaturazione}
Let $f \in C_{+}(\mathbb{R})$ be fixed. If 
\begin{equation}
\label{ipotesiopiccolo}
\| S_{w}^{\chi}(f) - f \|_{\infty} =o\left(\frac{1}{w}\right), \ \text{as} \ \ w \rightarrow +\infty,
\end{equation}
then $f$ is a constant function, i.e., the saturation order of the family $(S_{w}^{\chi})_{w>0}$ in $C_{+}(\mathbb{R})$ is $\frac{1}{w}$.
\end{teorema}
\begin{proof}
By assumption (\ref{ipotesiopiccolo}), there exists a function $g:\mathbb{R}\rightarrow\mathbb{R}^{+}$ with the property that $\lim_{w\rightarrow+\infty}g(w)=0$ and such that
\begin{equation*}
	\lvert S_{w}^{\chi}(f)(x)-f(x)\rvert \le \frac{g(w)}{w},
\end{equation*} 
for all $x \in \mathbb{R}$ and for all sufficiently large $w>0$. In particular, we have
\begin{equation*}
	\lvert S_{n}^{\chi}(f)(x)-f(x)\rvert \le \frac{g(n)}{n},
\end{equation*}
for all $x \in \mathbb{R}$ and for all sufficiently large $n \in \mathbb{N}$.
Let $\varepsilon>0$ be arbitrarily fixed. The condition $\lim_{n \rightarrow +\infty} g(n)=0$ implies that there exists $n_{0} \in \mathbb{N}$ such that $\frac{g(n)}{n}\le \frac{\varepsilon}{n}$ for all $n \in \mathbb{N}$, $n \ge n_{0}$. Hence, we have
\begin{equation*}
	\lvert S_{n}^{\chi}(f)(x)-f(x) \rvert \le \frac{\varepsilon}{n}, \ x \in \mathbb{R}, \ n \ge n_{0},
\end{equation*} 
and, in particular,
\begin{equation}
\label{equazione8}
	S_{n}^{\chi}(f)\left(\frac{2j+1}{2n}\right)-f\left(\frac{2j+1}{2n}\right)\le \frac{\varepsilon}{n}, \ j \in \mathbb{Z}, \ n \ge n_{0}.
\end{equation}
Let us now choose arbitrary $n \in \mathbb{N}$, $n \ge n_{0}$ and $j \in \mathbb{Z}$. By Lemma \ref{lemma1}, it follows that
\begin{equation*}
	\bigvee_{k \in \mathbb{Z}} \chi\left(n \cdot \frac{2j+1}{2n}-k\right)=\chi\left(n\cdot\frac{2j+1}{2n}-j\right)=\chi\left(n\cdot\frac{2j+1}{2n}-(j+1)\right).
\end{equation*}
Therefore, we can write:
\begin{equation*}
	\begin{split}
		S_{n}^{\chi}(f)\left(\frac{2j+1}{2n}\right)&=\frac{\displaystyle \bigvee_{k \in \mathbb{Z}}f\left(\frac{k}{n}\right)\chi\left(n \cdot \frac{2j+1}{2n}-k\right)}{\displaystyle \bigvee_{k \in \mathbb{Z}}\chi\left(n \cdot \frac{2j+1}{2n}-k\right)} \\
		& \ge \frac{\max\left\{f\left(\frac{j}{n}\right)\chi\left(n\cdot\frac{2j+1}{2n}-j\right), f\left(\frac{j+1}{n}\right)\chi\left(n\cdot\frac{2j+1}{2n}-(j+1)\right)\right\}}{\displaystyle \bigvee_{k \in \mathbb{Z}}\chi\left(n \cdot \frac{2j+1}{2n}-k\right)},
	\end{split}
\end{equation*}
which implies 
\begin{equation*}
	S_{n}^{\chi}(f)\left(\frac{2j+1}{2n}\right) \ge \max\left\{f\left(\frac{j}{n}\right),f\left(\frac{j+1}{n}\right)\right\}.
\end{equation*}
By combining this inequality with (\ref{equazione8}), we obtain
\begin{equation}
\label{equazione9}
	f\left(\frac{j}{n}\right)-f\left(\frac{2j+1}{2n}\right) \le \frac{\varepsilon}{n}, \ j \in \mathbb{Z}, \ n \ge n_{0},
\end{equation}
and 
\begin{equation}
\label{equazione10}
	f\left(\frac{j+1}{n}\right)-f\left(\frac{2j+1}{2n}\right) \le \frac{\varepsilon}{n}, \ j \in \mathbb{Z}, \ n \ge n_{0}.
\end{equation}
Now, applying successively (\ref{equazione9}) for fixed $n$ and $j$, we have
\begin{equation*}
	\begin{split}
		&f\left(\frac{j}{n}\right)-f\left(\frac{2j+1}{2n}\right) \le \frac{\varepsilon}{n}, \\
		&f\left(\frac{2j+1}{2n}\right)-f\left(\frac{4j+3}{4n}\right) \le \frac{\varepsilon}{2n}, \ (j:=2j+1, n:=2n) \\
		& \quad \quad \quad \quad \quad \quad \quad \quad \quad \quad  \cdot \\
		& \quad \quad \quad \quad \quad \quad \quad \quad \quad \quad \cdot \\
		&f\left(\frac{2^{k}j+2^{k}-1}{2^{k}n}\right)-f\left(\frac{2^{k+1}j+2^{k+1}-1}{2^{k+1}n}\right)\le \frac{\varepsilon}{2^{k}n}, \ (j:=2^{k}j+2^{k}-1, n:=2^{k}n),
	\end{split}
\end{equation*}
for all $k \in \mathbb{N}$. Taking the sum of the previous inequalities, we get
\begin{equation*}
	f\left(\frac{j}{n}\right)-f\left(\frac{2^{k+1}j+2^{k+1}-1}{2^{k+1}n}\right)\le \frac{\varepsilon}{n}\cdot\left(1+\frac{1}{2}+\frac{1}{2^{2}}+\cdots+\frac{1}{2^{k}}\right)\le \frac{2\varepsilon}{n}.
\end{equation*}
Passing to the limit as $k \rightarrow+\infty$ and using the continuity of $f$, we get
\begin{equation}
\label{equazione11}
	f\left(\frac{j}{n}\right)-f\left(\frac{j+1}{n}\right) \le \frac{2\varepsilon}{n}. 
\end{equation}
Now, applying successively (\ref{equazione10}) for fixed $n$ and $j$, we have
\begin{equation*}
	\begin{split}
		&f\left(\frac{j+1}{n}\right)-f\left(\frac{2j+1}{2n}\right) \le \frac{\varepsilon}{n}, \\
		&f\left(\frac{2j+1}{2n}\right)-f\left(\frac{4j+1}{4n}\right) \le \frac{\varepsilon}{2n}, \ (j:=2j, n:=2n) \\
		& \quad \quad \quad \quad \quad \quad \quad \quad \quad \quad  \cdot \\
		& \quad \quad \quad \quad \quad \quad \quad \quad \quad \quad \cdot \\
		&f\left(\frac{2^{k}j+1}{2^{k}n}\right)-f\left(\frac{2^{k+1}j+1}{2^{k+1}n}\right)\le \frac{\varepsilon}{2^{k}n}, \ (j:=2^{k}j, n:=2^{k}n),
	\end{split}
\end{equation*}
for all $k \in \mathbb{N}$. Taking the sum of the previous inequalities, we get
\begin{equation*}
	f\left(\frac{j+1}{n}\right)-f\left(\frac{2^{k+1}j+1}{2^{k+1}n}\right)\le \frac{\varepsilon}{n}\cdot\left(1+\frac{1}{2}+\frac{1}{2^{2}}+\cdots+\frac{1}{2^{k}}\right)\le \frac{2\varepsilon}{n}.
\end{equation*}
Passing to the limit as $k \rightarrow+\infty$ and using the continuity of $f$, we get
\begin{equation}
\label{equazione12}
	f\left(\frac{j+1}{n}\right)-f\left(\frac{j}{n}\right) \le \frac{2\varepsilon}{n}.
\end{equation}
From (\ref{equazione11}) and (\ref{equazione12}), we have
\begin{equation}
\label{equazione13}
	\left\lvert f\left(\frac{j+1}{n}\right)-f\left(\frac{j}{n}\right)\right\rvert \le \frac{2\varepsilon}{n}, \ j \in \mathbb{Z}, n \ge n_{0}. 
\end{equation}
Now fix $x_{0}$, $y_{0} \in \mathbb{R}$ with $x_{0}<y_{0}$. Since $f$ is continuous, there exists $n_{1}=n_{1}(\varepsilon) \in \mathbb{N}$ sufficiently large such that for every $x \in \mathbb{R}$, $\lvert x-x_{0} \rvert \le \frac{1}{n_{1}(\varepsilon)}$ and every $y \in \mathbb{R}$, $\lvert y-y_{0}\rvert \le \frac{1}{n_{1}(\varepsilon)}$, we have 
\begin{equation*}
	\lvert f(x)-f(x_{0})\rvert \le \varepsilon, \ \lvert f(y)-f(y_{0})\rvert \le \varepsilon. 
\end{equation*}
Let $n_{2}:=\max\{n_{0},n_{1},1/(y_{0}-x_{0})\}$ and let $n \in \mathbb{N}$, $n \ge n_{2}$ be fixed. The condition $1/n_{2}\le y_{0}-x_{0}$ guarantees the existence of $k \in \mathbb{Z}$ and $l \in \mathbb{N}$ such that 
\begin{equation*}
	\frac{k-1}{n} \le x_{0} \le \frac{k}{n} \le \frac{k+1}{n}\le \dots \le \frac{k+l}{n} \le y_{0} \le \frac{k+l+1}{n}.
\end{equation*}
Applying the triangle inequality repeatedly, we have
\begin{equation*}
	\begin{split}
		\lvert f(x_{0})-f(y_{0})\rvert \le &\left\lvert f(x_{0})-f\left(\frac{k}{n}\right)\right\rvert+\sum_{p=0}^{l-1} \left\lvert f\left(\frac{k+p}{n}\right)-f\left(\frac{k+p+1}{n}\right)\right\rvert \\
		&+\left\lvert f\left(\frac{k+l}{n}\right)-f(y_{0})\right\rvert.
	\end{split}
\end{equation*}
By (\ref{equazione13}), we have
\begin{equation*}
	\left\lvert f\left(\frac{k+p}{n}\right)-f\left(\frac{k+p+1}{n}\right)\right\rvert \le \frac{2 \varepsilon}{n}, \ p=0,1,\dots,l-1,
\end{equation*}
and noting that $\max\left\{\left\lvert x_{0}-\frac{k}{n}\right\rvert, \left\lvert y_{0}-\frac{k+l}{n}\right\rvert\right\} \le \frac{1}{n}\le \frac{1}{n_{1}(\varepsilon)}$, it follows that
\begin{equation*}
	\lvert f(x_{0})-f(y_{0})\rvert \le \frac{l 2\varepsilon}{n}+2\varepsilon.
\end{equation*} 
Moreover, from $\frac{k+l}{n}-\frac{k}{n} \le y_{0}-x_{0}$, we obtain $l \le n(y_{0}-x_{0})$, and hence
\begin{equation*}
	\lvert f(x_{0})-f(y_{0})\rvert \le \varepsilon(2(y_{0}-x_{0})+2).
\end{equation*} 
Now, since $\varepsilon>0$ was arbitrary, passing to the infimum for $\varepsilon>0$ in the previous inequality, we obtain $f(x_{0})=f(y_{0})$. By the arbitrariness of $x_{0}$ and $y_{0}$, it turns out that $f$ is constant. This completes the proof.
\end{proof}
The previous theorem shows that $\frac{1}{n}$ is the best possible order for uniform approximation processes based on the operators $S_{n}^{\chi}$ in (\ref{operatoriequazione2}) when $I=\mathbb{R}$, provided that they are constructed using a kernel function satisfying assumptions $(\chi1)-(\chi5)$. This means that if a non-negative, bounded, uniformly continuous function $f$ can be approximated by the max-product generalized sampling operators with a rate of convergence higher than $\frac{1}{n}$ as $n \rightarrow+\infty$, then $f$ must be constant on the whole real axis. Since it is easy to see that $S_{n}^{\chi}$ reproduces non-negative constant functions on $\mathbb{R}$, it follows that the saturation class of $S_{n}^{\chi}$ in $C_{+}(\mathbb{R})$ is exactly the class of non-negative constant functions. Note that, according to (\ref{ordineLip}), the saturation order is attained for non-negative Lipschitz functions on $\mathbb{R}$.
\newline We point out that Theorem \ref{teoremadisaturazione} holds for non-negative functions $f \in C(\mathbb{R})$. However, as in the case for several other max-product type operators (see, e.g., \cite{CoroianuGal2012,CoroianuGal2014,CostarelliVinti2017}), the above saturation result can be extended to bounded and uniformly continuous functions on $\mathbb{R}$ of variable sign, as shown in the following remark.
\begin{remark}
\label{Remark1}
It is well-known that, if $f \in C(\mathbb{R})$ has variable sign, then the family of non-linear operators $(S_{n}^{\chi}(f(\cdot)-m)+m)_{n \in \mathbb{N}}$ converges uniformly to $f$ on $\mathbb{R}$, where $m:=\inf_{x \in \mathbb{R}}f(x)$ (see Remark 3.4 of \cite{CoroianuCostarelli2019}). As a consequence of Theorem \ref{teoremadisaturazione}, if 
\begin{equation*}
\| S_{n}^{\chi}(f(\cdot)-m)+m-f(\cdot) \|_{\infty} = o\left(\frac{1}{n}\right), \quad \text{as} \ n \rightarrow +\infty,
\end{equation*}
then $f(\cdot)-m$ must be constant, which implies that $f$ itself is constant.
\end{remark}
\section{Local inverse result}
\label{Sezione4}
The quantitative estimate recalled in (\ref{Jackson-typeestimate}) constitutes a direct approximation result, as, by (\ref{ordineLip}), it  provides the order of uniform approximation  for the max-product generalized sampling operators in (\ref{equazione4saturazione}), provided that $f$ is a non-negative Lipschitz continuous function on $\mathbb{R}$. \newline In this section, we investigate whether a converse result holds. Unfortunately, as in the case for several max-product type operators (see, e.g., \cite{CoroianuGal2012,CoroianuGal2014,CostarelliVinti2017,CantariniCoroianu2021}), a global inverse result cannot be obtained; only a local version is possible. More precisely, the aim of this section is to prove the following local inverse result: if, for some $f \in C_{+}(\mathbb{R})$, there holds:
\begin{equation*}
\sup_{x \in [a,b]} \lvert S_{w}^{\chi}(f)(x)-f(x) \rvert = \mathcal{O}\left(\frac{1}{w}\right) \quad \text{as} \ w \rightarrow +\infty,
\end{equation*}
then $f \in \textnormal{Lip}([a,c])$ for any $c \in [a,b)$, whenever $0<a<b$, and $ f\in \textnormal{Lip}([c,b])$ for any $c \in (a,b]$, whenever $ a<b<0$. \newline To obtain this result, we need the following auxiliary lemmas.
\begin{lemma}
\label{lemma4.1}
	Let $a$, $b \in \mathbb{R}$ with $0<a<b$ and let $f:\mathbb{R}\rightarrow\mathbb{R}_{0}^{+}$ be a bounded function. For any $j \in \{0,1,\dots,n\}$, $n \in \mathbb{N}$, with $j \le \frac{n}{2}$, we have:
	\begin{equation*}
		\begin{split}
			&(i) \ \ S_{n/b}^{\chi}(f)\left(\frac{j}{(n+1)/b}\right) \ge f\left(\frac{j}{n/b}\right), \\
			&(ii) \ \ S_{(n+1)/b}^{\chi}(f)\left(\frac{j}{n/b}\right) \ge f\left(\frac{j}{(n+1)/b}\right).
		\end{split}
	\end{equation*}
\end{lemma}
\begin{proof}
	First, it is easy to observe that if $j \le \frac{n}{2}$, then we have $\frac{j}{(n+1)/b} \in \left[\frac{j-1}{n/b}+\frac{1/2}{n/b},\frac{j}{n/b}\right]$. Thus, by Lemma \ref{lemma1}, we have
	\begin{equation*}
		\bigvee_{k \in \mathbb{Z}} \chi\left(\frac{n}{b}\cdot\frac{j}{(n+1)/b}-k\right)=\chi\left(\frac{n}{b}\cdot\frac{j}{(n+1)/b}-j\right),
	\end{equation*}
	which implies that
	\begin{equation*}
		\begin{split}
			S_{n/b}^{\chi}(f)\left(\frac{j}{(n+1)/b}\right) &= \frac{\displaystyle \bigvee_{k \in \mathbb{Z}}f\left(\frac{k}{n/b}\right)\chi\left(\frac{n}{b}\cdot\frac{j}{(n+1)/b}-k\right)}{\displaystyle \chi\left(\frac{n}{b}\cdot\frac{j}{(n+1)/b}-j\right)} \\
			& \ge \frac{\displaystyle f\left(\frac{j}{n/b}\right)\chi\left(\frac{n}{b}\cdot\frac{j}{(n+1)/b}-j\right)}{\displaystyle \chi\left(\frac{n}{b}\cdot\frac{j}{(n+1)/b}-j\right)}=f\left(\frac{j}{n/b}\right). 
		\end{split}
	\end{equation*}
	Moreover, since $j \le \frac{n}{2}$, it immediately follows that $\frac{j}{n/b}\in \left[\frac{j}{(n+1)/b},\frac{j}{(n+1)/b}+\frac{1/2}{(n+1)/b}\right]$. Therefore, by Lemma \ref{lemma1}, we obtain
	\begin{equation*}
		\bigvee_{k \in \mathbb{Z}} \chi\left(\frac{n+1}{b}\cdot\frac{j}{n/b}-k\right)=\chi\left(\frac{n+1}{b}\cdot\frac{j}{n/b}-j\right).
	\end{equation*}
	Therefore, we can write the following:
	\begin{equation*}
		\begin{split}
			S_{(n+1)/b}^{\chi}(f)\left(\frac{j}{n/b}\right)&= \frac{\displaystyle \bigvee_{k \in \mathbb{Z}}f\left(\frac{k}{(n+1)/b}\right)\chi\left(\frac{n+1}{b}\cdot\frac{j}{n/b}-k\right)}{\displaystyle \chi\left(\frac{n+1}{b}\cdot\frac{j}{n/b}-j\right)} \\
			& \ge \frac{\displaystyle f\left(\frac{j}{(n+1)/b}\right)\chi\left(\frac{n+1}{b}\cdot\frac{j}{n/b}-j\right)}{\displaystyle \chi\left(\frac{n+1}{b}\cdot\frac{j}{n/b}-j\right)}=f\left(\frac{j}{(n+1)/b}\right).
		\end{split}
	\end{equation*}
	This completes the proof.
\end{proof}
Moreover, a similar result for the case $j \ge \frac{n}{2}$ can be established. 
\begin{lemma}
\label{lemma4.2}
	Let $a$, $b \in \mathbb{R}$, with $0<a<b$ and let $f:\mathbb{R}\rightarrow\mathbb{R}_{0}^{+}$ be a bounded function. For any $j \in \{0,1,\dots,n\}$, $n \in \mathbb{N}$, with $j \ge \frac{n}{2}$, we have:
	\begin{equation*}
		\begin{split}
			&(i) \ \ S_{n/b}^{\chi}(f)\left(\frac{j+1}{(n+1)/b}\right) \ge f\left(\frac{j}{n/b}\right), \\
			&(ii) \ \ S_{(n+1)/b}^{\chi}(f)\left(\frac{j}{n/b}\right) \ge f\left(\frac{j+1}{(n+1)/b}\right).
		\end{split}
	\end{equation*}
\end{lemma}
\begin{proof}
	The condition $j \ge \frac{n}{2}$ implies that
	\begin{equation*}
		\frac{j+1}{(n+1)/b}\in \left[\frac{j}{n/b},\frac{j}{n/b}+\frac{1/2}{n/b}\right],
	\end{equation*}
	and
	\begin{equation*}
		\frac{j}{n/b} \in \left[\frac{j}{(n+1)/b}+\frac{1/2}{(n+1)/b},\frac{j}{(n+1)/b}+\frac{1}{(n+1)/b}\right].
	\end{equation*}
	Thus, by reasoning as in the proof of Lemma \ref{lemma4.1}, we immediately obtain the thesis. 
\end{proof}
\begin{lemma}
	\label{lemma4.3}
	Let $a$, $b \in \mathbb{R}$ with $0<a<b$, and let $f \in C_{+}(\mathbb{R})$ with the property that there exists a constant $C>0$ such that for any $k$, $n \in \mathbb{N}$, $n \ge 1$, $k \le n$, with $a \le \frac{k}{(n+1)/b}<\frac{k}{n/b}\le b$, we have:
	\begin{equation*}
		\left\lvert f\left(\frac{k}{(n+1)/b}\right)-f\left(\frac{k}{n/b}\right)\right\rvert \le \frac{bC}{n}.
	\end{equation*}  
	Then $f$ is a Lipschitz function on $[a,b]$ with Lipschitz constant $\frac{bC}{a}$, i.e., 
	\begin{equation*}
		\lvert f(x)-f(y)\rvert \le \frac{bC}{a}\cdot \lvert x-y \rvert, \quad x,y \in [a,b].
	\end{equation*}
\end{lemma}
\begin{proof}
	Let $n \in \mathbb{N}$, with $n \ge 1$, and let $x$, $y \in [a,b]$ be arbitrarily fixed. Since $f \in C_{+}(\mathbb{R})$, it follows that there exists $m \in \mathbb{N}$ such that
	\begin{equation}
		\label{equazione15}
		\lvert f(u)-f(v) \rvert \le \frac{1}{n}, \ \text{for every} \ u, v \in [a,b]: \lvert u-v \rvert \le \frac{1}{m/b}.
	\end{equation}
	Furthermore, noting that the limit $k_{0}/((m+l+1)/b) \rightarrow0$ as $l \rightarrow+\infty$, and recalling that $x >0$, there exist $k_{0}$, $l_{0} \in \mathbb{N}$, with $l_{0}\ge 1$, such that
	\begin{equation*}
		\frac{k_{0}}{(m+l_{0}+1)/b}\le x \le \frac{k_{0}}{(m+l_{0})/b}<\dots<\frac{k_{0}}{(m+1)/b}\le y \le \frac{k_{0}}{m/b}.
	\end{equation*} 
	Since $y \le b$, we have $k_{0}\le m+1$. Moreover, note that 
	\begin{equation*}
		y -\frac{k_{0}}{(m+1)/b}\le \frac{k_{0}}{m/b}-\frac{k_{0}}{(m+1)/b}=\frac{k_{0}}{m(m+1)/b}\le \frac{1}{m/b}.
	\end{equation*} 
	Therefore, by (\ref{equazione15}), we obtain
	\begin{equation*}
		\left\lvert f(y)-f\left(\frac{k_{0}}{(m+1)/b}\right)\right\rvert\le \frac{1}{n}.
	\end{equation*}
	Similarly, we have
	\begin{equation*}
		\left\lvert f(x)-f\left(\frac{k_{0}}{(m+l_{0})/b}\right)\right\rvert \le \frac{1}{n}.
	\end{equation*}
	Then, we can write what follows:
	\begin{equation*}
		\begin{split}
			\lvert f(x)-f(y) \rvert &\le \left\lvert f(x)-f\left(\frac{k_{0}}{(m+l_{0})/b}\right)\right\rvert + \sum_{p=1}^{l_{0}-1}\left\lvert f\left(\frac{k_{0}}{(m+p+1)/b}\right)-f\left(\frac{k_{0}}{(m+p)/b}\right)\right\rvert \\
			& + \left \lvert f(y)-f\left(\frac{k_{0}}{(m+1)/b}\right) \right \rvert
		\end{split}
	\end{equation*}
	\begin{equation}
		\label{equazione16}
		\hskip-1.4cm \le \frac{2}{n}+\sum_{p=1}^{l_{0}-1}\left\lvert f\left(\frac{k_{0}}{(m+p+1)/b}\right)-f\left(\frac{k_{0}}{(m+p)/b}\right)\right\rvert.
	\end{equation}
	Now, by the assumption, we get
	\begin{equation*}
		\left\lvert f\left(\frac{k_{0}}{(m+p+1)/b}\right)-f\left(\frac{k_{0}}{(m+p)/b}\right)\right\rvert \le \frac{bC}{m+p},
	\end{equation*}
	for any $p \in \{1,2,\dots,l_{0}-1\}$. Moreover, we have
	\begin{equation*}
		\begin{split}
			\frac{k_{0}}{(m+p)/b}-\frac{k_{0}}{(m+p+1)/b}&=\frac{bk_{0}}{(m+p)(m+p+1)}=\frac{b}{m+p}\cdot \frac{k_{0}}{(m+p+1)/b}\cdot\frac{1}{b} \\
			& \ge \frac{b}{m+p}\cdot \frac{a}{b},
		\end{split}
	\end{equation*}
	for every $p \in \{1,2,\dots,l_{0}-1\}$. From the last two inequalities, it follows that
	\begin{equation*}
		\left\lvert f\left(\frac{k_{0}}{(m+p+1)/b}\right)-f\left(\frac{k_{0}}{(m+p)/b}\right)\right\rvert \le \left(\frac{k_{0}}{(m+p)/b}-\frac{k_{0}}{(m+p+1)/b}\right)\cdot \frac{bC}{a}.
	\end{equation*}
	Using the previous inequality in (\ref{equazione16}), we obtain
	\begin{equation*}
		\begin{split}
			\lvert f(x)-f(y) \rvert &\le \frac{2}{n}+\sum_{p=1}^{l_{0}-1}\left\lvert f\left(\frac{k_{0}}{(m+p+1)/b}\right)-f\left(\frac{k_{0}}{(m+p)/b}\right)\right\rvert \\
			& \le \frac{2}{n}+\frac{bC}{a} \sum_{p=1}^{l_{0}-1}\left(\frac{k_{0}}{(m+p)/b}-\frac{k_{0}}{(m+p+1)/b}\right) \\
			&=\frac{2}{n}+\frac{bC}{a}\left(\frac{k_{0}}{(m+1)/b}-\frac{k_{0}}{(m+l_{0})/b}\right) \le \frac{2}{n}+\frac{bC}{a}(y-x).
		\end{split}
	\end{equation*}
	Finally, since 
	\begin{equation*}
		\lvert f(x)-f(y) \rvert \le \frac{2}{n}+\frac{bC}{a}\lvert x-y\rvert,
	\end{equation*}
	for every $n \in \mathbb{N}$, letting $n \rightarrow+\infty$, we obtain
	\begin{equation*}
		\lvert f(x)-f(y) \rvert \le \frac{bC}{a}\lvert x-y\rvert.
	\end{equation*}
	Since $x$ and $y$ are arbitrary points in $[a,b]$, the proof is complete.  
\end{proof}
\begin{lemma}
	\label{lemma4.4}
	Let $a$, $b \in \mathbb{R}$ with $0<a<b$, and let $f \in C_{+}(\mathbb{R})$ with the property that there exists a constant $C>0$ such that for any $k$, $n \in \mathbb{N}$, $n \ge 1$, $k \le n$, with $a \le \frac{k}{n/b}<\frac{k+1}{(n+1)/b}\le b$, we have:
	\begin{equation*}
		\left\lvert f\left(\frac{k}{n/b}\right)-f\left(\frac{k+1}{(n+1)/b}\right)\right\rvert \le \frac{bC}{n}.
	\end{equation*}  
	Then $f$ is a Lipschitz function on any interval $[a,c]$ with Lipschitz constant $\frac{bC}{b-c}$, i.e., 
	\begin{equation*}
		\lvert f(x)-f(y)\rvert \le \frac{bC}{b-c}\cdot \lvert x-y \rvert, \quad x,y \in [a,c],
	\end{equation*}
	where $c \in [a,b)$ is arbitrarily chosen. 
\end{lemma}
\begin{proof}
	Let $n\in \mathbb{N}$, with $n \ge 1$, and fix arbitrary $x$, $y \in [a,c]$, where $a\le c<b$. As in the proof of Lemma \ref{lemma4.3}, there exists a sufficiently large $m \in \mathbb{N}$ such that (\ref{equazione15}) holds. Since $x>0$ and $\frac{k_{0}+l+1}{(m+l+1)/b}\rightarrow b>y$ as $l\rightarrow+\infty$, there exist $k_{0}$, $l_{0}\in \mathbb{N}$, with $l_{0}\ge 1$, such that
	\begin{equation*}
		\frac{k_{0}}{m/b}\le x \le \frac{k_{0}+1}{(m+1)/b}\le \dots \le \frac{k_{0}+l_{0}}{(m+l_{0})/b}\le y \le \frac{k_{0}+l_{0}+1}{(m+l_{0}+1)/b}.
	\end{equation*}
	Observe that $k_{0}\le m$. Moreover, note that
	\begin{equation*}
		\frac{k_{0}+1}{(m+1)/b}-x \le \frac{k_{0}+1}{(m+1)/b} -\frac{k_{0}}{m/b}=\frac{m-k_{0}}{m(m+1)/b}<\frac{1}{m/b}.
	\end{equation*}
	Therefore, by (\ref{equazione15}), we obtain
	\begin{equation*}
		\left\lvert f(x)-f\left(\frac{k_{0}+1}{(m+1)/b}\right) \right\rvert \le \frac{1}{n}.
	\end{equation*}
	Similarly, we have
	\begin{equation*}
		\left\lvert f(y)-f\left(\frac{k_{0}+l_{0}}{(m+l_{0})/b}\right)\right\rvert \le \frac{1}{n}.
	\end{equation*}
	Then, we can write what follows:
	\begin{equation*}
		\begin{split}
			\lvert f(x)-f(y) \rvert &\le \left\lvert f(x)-f\left(\frac{k_{0}+1}{(m+1)/b}\right)\right\rvert + \sum_{p=1}^{l_{0}-1}\left\lvert f\left(\frac{k_{0}+p}{(m+p)/b}\right)-f\left(\frac{k_{0}+p+1}{(m+p+1)/b}\right)\right\rvert \\
			& + \left \lvert f(y)-f\left(\frac{k_{0}+l_{0}}{(m+l_{0})/b}\right) \right \rvert
		\end{split}
	\end{equation*}
	\begin{equation}
	\label{equazione17}
		\hskip-1.4cm \le \frac{2}{n}+\sum_{p=1}^{l_{0}-1}\left\lvert f\left(\frac{k_{0}+p}{(m+p)/b}\right)-f\left(\frac{k_{0}+p+1}{(m+p+1)/b}\right)\right\rvert.
	\end{equation}
	Now, by the assumption, we get
	\begin{equation*}
		\left\lvert f\left(\frac{k_{0}+p}{(m+p)/b}\right)-f\left(\frac{k_{0}+p+1}{(m+p+1)/b}\right)\right\rvert \le \frac{bC}{m+p},
	\end{equation*}
	for any $p \in \{1,2,\dots,l_{0}-1\}$. Moreover, we have
	\begin{equation}
	\label{equazione18}
		\begin{split}
			\frac{k_{0}+p+1}{(m+p+1)/b}-\frac{k_{0}+p}{(m+p)/b}&=\frac{b(m-k_{0})}{(m+p)(m+p+1)}=\frac{b}{m+p}\cdot \frac{m-k_{0}}{m+p+1},
		\end{split}
	\end{equation}
	for every $p \in \{1,2,\dots,l_{0}-1\}$. Now, since 
	\begin{equation*}
		\frac{k_{0}+p}{(m+p)/b}<y \le c,
	\end{equation*}
	we get 
	\begin{equation*}
		k_{0}<\frac{mc}{b}+p\left(\frac{c}{b}-1\right),
	\end{equation*}
	which implies 
	\begin{equation*}
		\frac{m-k_{0}}{m+p+1}\ge \frac{m-\frac{mc}{b}-p\left(\frac{c}{b}-1\right)}{m+p+1}=\frac{\left(1-\frac{c}{b}\right)(m+p)}{m+p+1}\ge \left(1-\frac{c}{b}\right)\cdot\frac{m}{m+1}.
	\end{equation*}
	Using the previous inequality in (\ref{equazione18}), we obtain
	\begin{equation*}
		\frac{k_{0}+p+1}{(m+p+1)/b}-\frac{k_{0}+p}{(m+p)/b} \ge \frac{b}{m+p}\cdot \left(1-\frac{c}{b}\right)\cdot\frac{m}{m+1}.
	\end{equation*}
	From the previous inequality, it follows that
	\begin{equation*}
		\begin{split}
			&\left\lvert f\left(\frac{k_{0}+p}{(m+p)/b}\right)-f\left(\frac{k_{0}+p+1}{(m+p+1)/b}\right)\right\rvert \le \frac{bC}{m+p} \\
			& \le C \cdot \frac{m+1}{m}\cdot\frac{b}{b-c}\left(\frac{k_{0}+p+1}{(m+p+1)/b}-\frac{k_{0}+p}{(m+p)/b}\right).
		\end{split}
	\end{equation*}
	Hence, estimating $\lvert f(x)-f(y)\rvert$ as in (\ref{equazione17}), we have
	\begin{equation*}
		\begin{split}
			\lvert f(x)-f(y) \rvert & \le \frac{2}{n}+\sum_{p=1}^{l_{0}-1}\left\lvert f\left(\frac{k_{0}+p}{(m+p)/b}\right)-f\left(\frac{k_{0}+p+1}{(m+p+1)/b}\right)\right\rvert \\
			&\le \frac{2}{n}+\frac{Cb(m+1)}{m(b-c)}\cdot \sum_{p=1}^{l_{0}-1}\left(\frac{k_{0}+p+1}{(m+p+1)/b}-\frac{k_{0}+p}{(m+p)/b}\right) \\
			&=\frac{2}{n}+\frac{Cb(m+1)}{m(b-c)}\cdot\left(\frac{k_{0}+l_{0}}{(m+l_{0})/b}-\frac{k_{0}+1}{(m+1)/b}\right) \\
			&\le \frac{2}{n}+\frac{Cb(m+1)}{m(b-c)}(y-x).
		\end{split}
	\end{equation*}
	Without loss of generality, we may assume that $m \ge n$. Therefore, by letting $n \rightarrow +\infty$, and hence $m\rightarrow+\infty$, we finally obtain
	\begin{equation*}
		\lvert f(x)-f(y) \rvert \le \frac{Cb}{b-c}\cdot\lvert x-y \rvert.
	\end{equation*}
	Since $x$ and $y$ are arbitrary points in $[a,c]$, the proof is complete. 
\end{proof}
Now, we are in a position to prove a local inverse theorem for the operators $S_{w}^{\chi}$ in (\ref{equazione4saturazione}).
\begin{teorema}
\label{teorema4.1}
	Let $a$, $b \in \mathbb{R}$, with $a<b$, and let $f \in C_{+}(\mathbb{R})$. Suppose that there exists a constant $M>0$, independent of $w>0$ but depending on $f$, $a$ and $b$, such that 
	\begin{equation*}
		\sup_{x \in [a,b]} \lvert S_{w}^{\chi}(f)(x)-f(x)\rvert \le \frac{M}{w},
	\end{equation*}
	for all sufficiently large $w>0$. Then:
	\begin{description}
		\item[$(i)$] if $0<a<b$, then, for every $c \in [a,b)$, $f$ is a Lipschitz function on $[a,c]$ with Lipschitz constant $Mb\left(\frac{1}{2a}+\frac{1}{b-c}\right)$;
		\item[$(ii)$] if $a<b<0$, then, for every $c \in (a,b]$, $f$ is a Lipschitz function on $[c,b]$ with Lipschitz constant $Ma\left(\frac{1}{2b}-\frac{1}{c-a}\right)$.  
	\end{description}
\end{teorema}
\begin{proof}
	First, we consider the case $0<a<b$. Let $c \in [a,b)$ be fixed. We distinguish three sub-cases: A) $c \le \frac{b}{2}$; B) $a \ge \frac{b}{2}$; C) $a<\frac{b}{2}<c$. \newline Case A) In this case, we prove that $f$ is Lipschitz continuous on $[a,\frac{b}{2}]$. To this end, let $n \in \mathbb{N}$ be sufficiently large, and let $k \in \mathbb{N}$ such that 
	\begin{equation*}
		a \le \frac{k}{(n+1)/b}<\frac{k}{n/b}\le \frac{b}{2}.
	\end{equation*}
	Now, if $f\left(\frac{k}{(n+1)/b}\right)\le f\left(\frac{k}{n/b}\right)$, applying Lemma \ref{lemma4.1} $(i)$ (note that $k \le \frac{n}{2}$), then we get
	\begin{equation*}
		S_{n/b}^{\chi}(f)\left(\frac{k}{(n+1)/b}\right)-f\left(\frac{k}{(n+1)/b}\right)\ge f\left(\frac{k}{n/b}\right)-f\left(\frac{k}{(n+1)/b}\right)\ge 0.
	\end{equation*}
	By the assumption, we have
	\begin{equation*}
		S_{n/b}^{\chi}(f)\left(\frac{k}{(n+1)/b}\right)-f\left(\frac{k}{(n+1)/b}\right) \le \frac{bM}{n},
	\end{equation*}
	and hence we obtain
	\begin{equation*}
		0 \le f\left(\frac{k}{n/b}\right)-f\left(\frac{k}{(n+1)/b}\right) \le \frac{bM}{n}.
	\end{equation*}
	Similarly, if  $f\left(\frac{k}{(n+1)/b}\right)\ge f\left(\frac{k}{n/b}\right)$, by Lemma \ref{lemma4.1} $(ii)$, we have
	\begin{equation*}
		S_{(n+1)/b}^{\chi}(f)\left(\frac{k}{n/b}\right)-f\left(\frac{k}{n/b}\right)\ge f\left(\frac{k}{(n+1)/b}\right)-f\left(\frac{k}{n/b}\right)\ge 0.
	\end{equation*}
	The assumption implies that 
	\begin{equation*}
		S_{(n+1)/b}^{\chi}(f)\left(\frac{k}{n/b}\right)-f\left(\frac{k}{n/b}\right) \le \frac{bM}{n+1},
	\end{equation*}
	and hence we obtain
	\begin{equation*}
		0 \le f\left(\frac{k}{(n+1)/b}\right)-f\left(\frac{k}{n/b}\right) \le \frac{bM}{n+1}\le \frac{bM}{n}.
	\end{equation*}
	In conclusion it follows that 
	\begin{equation*}
		\left \lvert f\left(\frac{k}{(n+1)/b}\right)-f\left(\frac{k}{n/b}\right) \right \rvert \le \frac{bM}{n}.
	\end{equation*}
	Therefore, applying Lemma \ref{lemma4.3} with $C:=M$ and $b:=\frac{b}{2}$, we obtain
	\begin{equation*}
		\lvert f(x)-f(y)\rvert \le \frac{bM}{2a}\lvert x-y \rvert,
	\end{equation*}
	for every $x$, $y \in [a,\frac{b}{2}]$. Hence, $f$ is Lipschitz continuous on $[a,\frac{b}{2}]$ with Lipschitz constant $bM/2a$. Since $c \le \frac{b}{2}$, it follows that $f$ is also Lipschitz continuous on $[a,c]$ with the same Lipschitz constant. \newline Case B) Let $n \in \mathbb{N}$ be sufficiently large, and let $k \in \mathbb{N}$ such that
	\begin{equation*}
		a \le \frac{k}{n/b}<\frac{k+1}{(n+1)/b}\le b.
	\end{equation*}
	Note that the assumption $a \ge \frac{b}{2}$ implies that $k \ge \frac{n}{2}$. Moreover, we have $k \le n$.
	\newline Now, if $f\left(\frac{k}{n/b}\right)\ge f\left(\frac{k+1}{(n+1)/b}\right)$, applying Lemma \ref{lemma4.2} $(i)$, then we get
	\begin{equation*}
		S_{n/b}^{\chi}(f)\left(\frac{k+1}{(n+1)/b}\right)-f\left(\frac{k+1}{(n+1)/b}\right)\ge f\left(\frac{k}{n/b}\right)-f\left(\frac{k+1}{(n+1)/b}\right)\ge 0.
	\end{equation*}
	By the assumption, we have
	\begin{equation*}
		S_{n/b}^{\chi}(f)\left(\frac{k+1}{(n+1)/b}\right)-f\left(\frac{k+1}{(n+1)/b}\right) \le \frac{bM}{n},
	\end{equation*}
	and hence we obtain
	\begin{equation*}
		0 \le f\left(\frac{k}{n/b}\right)-f\left(\frac{k+1}{(n+1)/b}\right) \le \frac{bM}{n}.
	\end{equation*}
	Similarly, if  $f\left(\frac{k}{n/b}\right)\le f\left(\frac{k+1}{(n+1)/b}\right)$, by Lemma \ref{lemma4.2} $(ii)$, we have
	\begin{equation*}
		S_{(n+1)/b}^{\chi}(f)\left(\frac{k}{n/b}\right)-f\left(\frac{k}{n/b}\right)\ge f\left(\frac{k+1}{(n+1)/b}\right)-f\left(\frac{k}{n/b}\right)\ge 0.
	\end{equation*}
	The assumption implies that 
	\begin{equation*}
		S_{(n+1)/b}^{\chi}(f)\left(\frac{k}{n/b}\right)-f\left(\frac{k}{n/b}\right) \le \frac{bM}{n+1},
	\end{equation*}
	and hence we obtain
	\begin{equation*}
		0 \le f\left(\frac{k+1}{(n+1)/b}\right)-f\left(\frac{k}{n/b}\right) \le \frac{bM}{n+1}\le \frac{bM}{n}.
	\end{equation*}
	In conclusion it follows that 
	\begin{equation*}
		\left \lvert f\left(\frac{k}{n/b}\right)-f\left(\frac{k+1}{(n+1)/b}\right) \right \rvert \le \frac{bM}{n}.
	\end{equation*}
	Therefore, applying Lemma \ref{lemma4.4} with $C:=M$, we obtain
	\begin{equation*}
		\lvert f(x)-f(y)\rvert \le \frac{bM}{b-c}\lvert x-y \rvert,
	\end{equation*}
	for every $x$, $y \in [a,c]$. Hence, $f$ is Lipschitz continuous on $[a,c]$ with Lipschitz constant $bM/(b-c)$. \newline Case C) From Case A), we conclude that $f$ is Lipschitz continuous on $[a,\frac{b}{2}]$ with Lipschitz constant $bM/2a$. Applying the result of Case B) with $a:=\frac{b}{2}$, we deduce that $f$ is Lipschitz continuous on $[\frac{b}{2},c]$ with Lipschitz constant $bM/(b-c)$. Now, it is easy to verify that $f$ is a Lipschitz function on $[a,c]$ with Lipschitz constant $Mb(1/2a+1/(b-c))$. This completes the proof of the case $0<a<b$. \newline It remains to discuss the case $a<b<0$. Let us define the function 
	\begin{equation*}
		\begin{split}
			&g:\ \mathbb{R}\rightarrow\mathbb{R}^{+}_{0}, \\
			&g(x):=f(-x).
		\end{split}
	\end{equation*} 
	Since $\chi(wx-k)=\chi(w(-x)+k)$ for every $x \in \mathbb{R}$ and $k \in \mathbb{Z}$ by $(\chi4)$, we obtain:
	\begin{equation*}
		\begin{split}
			S_{w}^{\chi}(f)(x)&=\frac{\displaystyle \bigvee_{k \in \mathbb{Z}}f\left(\frac{k}{w}\right)\chi(wx-k)}{\displaystyle \bigvee_{k \in \mathbb{Z}}\chi(wx-k)}=\frac{\displaystyle \bigvee_{k \in \mathbb{Z}}g\left(\frac{-k}{w}\right)\chi(w(-x)+k)}{\displaystyle \bigvee_{k \in \mathbb{Z}}\chi(w(-x)+k)}
			\\ & =\frac{\displaystyle \bigvee_{k \in \mathbb{Z}}g\left(\frac{k}{w}\right)\chi(w(-x)-k)}{\displaystyle \bigvee_{k \in \mathbb{Z}}\chi(w(-x)-k)}=S_{w}^{\chi}(g)(-x),
		\end{split}
	\end{equation*}
	for every $x \in \mathbb{R}$. Therefore, it follows that 
	\begin{equation*}
		\lvert S_{w}^{\chi}(f)(x)-f(x)\rvert = \lvert S_{w}^{\chi}(g)(-x)-g(-x)\rvert.
	\end{equation*}
	Applying the result of the previous case with $f:=g$, $a:=-b$, $b:=-a$, and $c:=-c$, we conclude that $g$ is Lipschitz continuous on $[-b,-c]$ with Lipschitz constant $M(-a)(1/(2(-b))+1/(-a+c))$. Consequently, $f$ is a Lipschitz function on $[c,b]$ with Lipschitz constant $Ma(1/2b-1/(c-a))$. This completes the proof.
\end{proof}
\begin{remark}
As in the case of Theorem \ref{teoremadisaturazione} (see Remark \ref{Remark1}), the above theorem can be generalized to functions of arbitrary sign. More precisely, let $f \in C(\mathbb{R})$ be a function such that
\begin{equation*}
	\sup_{x \in [a,b]}\lvert S_{w}^{\chi}(f(\cdot)-m)(x)+m-f(x)\rvert=\mathcal{O}\left(\frac{1}{w}\right) \quad \text{as} \ w \rightarrow+\infty,
\end{equation*}
where $m=\inf_{x \in \mathbb{R}}f(x)$. Then, if $0<a<b$, it follows that $f(\cdot)-m \in \text{Lip}([a,c])$, and hence $f \in \text{Lip}([a,c])$, for every $c \in [a,b)$. Similarly, if $a<b<0$,  then $f(\cdot)-m \in \text{Lip}([c,b])$, and hence $f \in \text{Lip}([c,b])$, for every $c \in (a,b]$. 
\end{remark}
\section{Strong localization result}
\label{Sezione5}
In \cite{CoroianuGal2012}, the authors established a strong localization result for truncated max-product sampling operators based on specific kernels, such as the sinc/Whittaker and Fejér kernels. In this section, we show that an analogous result holds for the (truncated) max-product generalized sampling operators in (\ref{operatoriequazione2}) with $I=[0,1]$, provided that the even centered bell-shaped kernel function is strictly positive. More precisely, we prove that, by working with strictly positive kernels satisfying conditions $(\chi1)$ and $(\chi3)-(\chi5)$, if $f$ and $g$ are bounded functions with strictly positive lower bounds and coincide on a given subinterval $[a,b] \subset [0,1]$, then for sufficiently large $n \in \mathbb{N}$,
\begin{equation*}
S_{n}^{\chi}(f)(x)=\frac{\displaystyle\bigvee_{k=0}^{n}f\left(\frac{k}{n}\right)\chi(nx-k)}{\displaystyle \bigvee_{k=0}^{n}\chi(nx-k)}, \quad x \in [0,1],
\end{equation*}
and 
\begin{equation*}
S_{n}^{\chi}(g)(x)=\frac{\displaystyle\bigvee_{k=0}^{n}g\left(\frac{k}{n}\right)\chi(nx-k)}{\displaystyle \bigvee_{k=0}^{n}\chi(nx-k)}, \quad x \in [0,1],
\end{equation*}
coincide on every subinterval $[c,d]$ sufficiently close to $[a,b]$. We point out that strong localization results constitute a well-known feature, along with good rates of approximation, of the max-product versions of several discrete operators, such as Bernstein operators \cite{CoroianuGal2014bis} and several interpolation operators \cite{CoroianuGal2013bis}. Theorem \ref{Teorema5.1} extends the class of max-product operators exhibiting good localization properties. 
\newline To establish this, we first need the following crucial lemma. 
\begin{lemma}
\label{lemma5.1}
Let $\chi$ be a strictly positive kernel satisfying assumptions $(\chi1)$ and $(\chi3)-(\chi5)$, and let $f:[0,1]\rightarrow \mathbb{R}^{+}$ be a bounded function such that
\begin{equation*}
m_{f}:=\inf_{x \in [0,1]} f(x)>0.
\end{equation*}
Furthermore, let $a,b \in [0,1]$, with $0<a<b<1$, be fixed. Then, for all $c$, $d \in [a,b]$ satisfying $a<c<d<b$, there exists $N_{0} \in \mathbb{N}$, depending only on $a$, $b$, $c$, $d$, $f$ and $\chi$, such that for any $n \ge N_{0}$ and $x \in [c,d]$, we have:
\begin{equation*}
\bigvee_{k=0}^{n}\chi(nx-k)f\left(\frac{k}{n}\right)=\bigvee_{k \in I_{n,x}}\chi(nx-k)f\left(\frac{k}{n}\right),
\end{equation*}
where $I_{n,x}:=\{k=0,1,\dots,n: j_{x}-\sqrt{n}\le k \le j_{x}+\sqrt{n}\}$, and $j_{x}\in \{0,1,\dots,n-1\}$ is such that $x \in \left[\frac{j_{x}}{n},\frac{j_{x}}{n}+\frac{1}{n}\right]$.
\end{lemma}

\begin{proof}
Let $x \in [c,d]$, with $a<c<d<b$, and let $n \in \mathbb{N}$ be fixed. Clearly, there exists a suitable index $j_{x}\in \{0,1,\dots,n-1\}$, depending on both $x$ and $n$, such that $x \in \left[\frac{j_{x}}{n},\frac{j_{x}}{n}+\frac{1}{n}\right]$. Now, let us choose $n_{0} \in \mathbb{N}$ sufficiently large such that $n_{0} > \max\left\{\frac{1}{c-a},\frac{1}{b-d}\right\}$. Then, for every $n \ge n_{0}$, since $x \in [c,d] \cap \left[\frac{j_{x}}{n},\frac{j_{x}}{n}+\frac{1}{n}\right]$, and $a<c<d<b$, it follows that $a< \frac{j_{x}}{n}<b$. Indeed, suppose that there exists $n \ge n_{0}$ such that $\frac{j_{x}}{n} \le a$. Then, $c-\frac{j_{x}}{n} \ge c-a>\frac{1}{n}$, which implies $d>c>\frac{j_{x}}{n}+\frac{1}{n}>\frac{j_{x}}{n}$, in contradiction with the fact that $x \in [c,d] \cap \left[\frac{j_{x}}{n}, \frac{j_{x}}{n}+\frac{1}{n}\right]$. A similar contradiction arises if $\frac{j_{x}}{n} \ge b$. \newline Therefore, we obtain $na <j_{x}<nb$ for all $n \ge n_{0}$. This implies that for every $x \in [c,d]$, there exists $\alpha_{x}\in [a,b]$ such that $j_{x}=n\alpha_{x}$. It is worth noting that $n_{0}$ does not depend on $x$. 
\newline In what follows, we consider the sequence $(a_{n})_{n \in \mathbb{N}}$, defined by $a_{n}:=\sqrt{n}$. Clearly, there exists $n_{1}\in \mathbb{N}$ such that $na-a_{n}>0$, for all $n \ge n_{1}$. Let now $n \ge \max\{n_{0},n_{1}\}$. Note that, since $j_{x}>na$, it follows that $j_{x}-a_{n}>0$. 
\newline In order to obtain the desired conclusion, fix $k \in \{0,1,\dots,n\}\setminus I_{n,x}$; that is, either $k < j_{x}-a_{n}$ or $k>j_{x}+a_{n}$. 
\newline If $k<j_{x}-a_{n}$, we consider two cases: $(i)$ $ x \in \left[\frac{j_{x}}{n},\frac{j_{x}}{n}+\frac{1}{2n}\right]$, and $(ii)$ $ x \in \left[\frac{j_{x}}{n}+\frac{1}{2n}, \frac{j_{x}}{n}+\frac{1}{n}\right]$. We begin with the first one, i.e., $(i)$ $x \in \left[\frac{j_{x}}{n}, \frac{j_{x}}{n}+\frac{1}{2n}\right]$. Since $0 \le nx-j_{x}\le \frac{1}{2}$, by property $(\chi5)$, we get 
\begin{equation}
\label{equazione19}
\chi(nx-j_{x}) \ge \chi\left(\frac{1}{2}\right).
\end{equation}
Moreover, since $j_{x}\le nx \le j_{x}+\frac{1}{2}$ and $j_{x}-k>a_{n}$, it follows that $nx-k \ge j_{x}-k >a_{n}>0$. Thus, using property $(\chi5)$ again, we obtain $\chi(nx-k) \le \chi(a_{n})$, and hence: 
\begin{equation}
\label{equazione20}
	\frac{1}{\chi(nx-k)}\ge \frac{1}{\chi(a_{n})}. 
\end{equation} 
Now, by exploiting (\ref{equazione19}) and (\ref{equazione20}), and recalling that $\chi(y)>0$ for all $y \in \mathbb{R}$, we can write:
\begin{equation*}
\begin{split}
\frac{\chi(nx-j_{x})f\left(\frac{j_{x}}{n}\right)}{\chi(nx-k)f\left(\frac{k}{n}\right)} \ge \frac{\chi\left(\frac{1}{2}\right) f\left(\frac{j_{x}}{n}\right)}{\chi(nx-k)f\left(\frac{k}{n}\right)} \ge \frac{\chi\left(\frac{1}{2}\right)}{\chi(a_{n})}\cdot \frac{f\left(\frac{j_{x}}{n}\right)}{f\left(\frac{k}{n}\right)} \ge \frac{\chi\left(\frac{1}{2}\right)}{\chi(a_{n})}\cdot \frac{m_{f}}{M_{f}},
\end{split}
\end{equation*}
where $m_{f}$ and $M_{f}$ denote the infimum and the supremum of $f$ on $[0,1]$, respectively. Recall that, by assumption, both $m_{f}$ and $M_{f}$ are strictly positive real numbers. Since $\lim_{y\rightarrow+\infty}\chi(y)=0^{+}$ by $(\chi3)$, it follows that
\begin{equation*}
\lim_{n\rightarrow+\infty}\frac{\chi\left(\frac{1}{2}\right)}{\chi(a_{n})}\cdot\frac{m_{f}}{M_{f}}=+\infty.
\end{equation*}
Hence, there exists $n_{2} \in \mathbb{N}$, independent of $x \in [c,d]$ but depending on $f$ and $\chi$, with $n_{2} \ge \max\{n_{0},n_{1}\}$, such that
\begin{equation}
\label{equazione21}
\frac{\chi(nx-j_{x})f\left(\frac{j_{x}}{n}\right)}{\chi(nx-k)f\left(\frac{k}{n}\right)}>1, \  x \in [c,d],
\end{equation} 
for every $n \ge n_{2}$, and every $k \in \{0,1,\dots,n\}$ such that $k <j_{x}-a_{n}$ (i.e., $k \notin I_{n,x}$). \newline Now, consider the case $(ii)$ $x \in \left[\frac{j_{x}}{n}+\frac{1}{2n}, \frac{j_{x}}{n}+\frac{1}{n}\right]$. Since $-\frac{1}{2}\le nx-(j_{x}+1)\le 0$, by properties $(\chi4)$ and $(\chi5)$, we get:
\begin{equation*}
\chi(nx-(j_{x}+1)) \ge \chi\left(\frac{1}{2}\right). 
\end{equation*} 
Moreover, since $j_{x}+\frac{1}{2} \le nx \le j_{x}+1$ and $j_{x}-k>a_{n}$, it is immediate that $nx-k \ge j_{x}-k+\frac{1}{2} > a_{n}+\frac{1}{2} >a_{n}>0$, which implies that:
\begin{equation*}
\chi(nx-k) \le \chi\left(j_{x}-k+\frac{1}{2}\right) \le \chi\left(a_{n}+\frac{1}{2}\right) \le \chi(a_{n}).
\end{equation*}
Therefore, reasoning as in the proof of case $(i)$, it follows that for sufficiently large $n \in \mathbb{N}$ we have:
\begin{equation}
\label{equazione22}
\frac{\chi(nx-(j_{x}+1))f\left(\frac{j_{x}+1}{n}\right)}{\chi(nx-k)f\left(\frac{k}{n}\right)}>1, \ \ x \in [c,d], \ k \in \{0,1,\dots,n\}: k<j_{x}-a_{n}.
\end{equation}
From (\ref{equazione21}) and (\ref{equazione22}), we conclude that there exists $N_{1} \in \mathbb{N}$, depending only on $a$, $b$, $c$, $d$, $f$ and $\chi$, such that
\begin{equation*}
\frac{\max\left\{\chi(nx-j_{x})f\left(\frac{j_{x}}{n}\right),\chi(nx-(j_{x}+1))f\left(\frac{j_{x}+1}{n}\right)\right\}}{\chi(nx-k)f\left(\frac{k}{n}\right)}>1,
\end{equation*}
for all $x \in [c,d]$, $n \ge N_{1}$, and $k \in \{0,1,\dots,n\}$ such that $k < j_{x}-a_{n}$.  
\newline Now, consider the case $k>j_{x}+a_{n}$. As in the first part of the proof, we distinguish between two subcases: $(i^{*})$ $x \in \left[\frac{j_{x}}{n}, \frac{j_{x}}{n}+\frac{1}{2n}\right]$, and $(ii^{*})$ $x \in \left[\frac{j_{x}}{n}+\frac{1}{2n}, \frac{j_{x}}{n}+\frac{1}{n}\right]$. Proceeding with a similar argument and with slightly modifications with respect to the case $k<j_{x}-a_{n}$, it turns out that there exists $N_{2} \in \mathbb{N}$, depending only on $a$, $b$, $c$, $d$, $f$ and $\chi$, such that
\begin{equation*}
\frac{\max\left\{\chi(nx-j_{x})f\left(\frac{j_{x}}{n}\right), \chi(nx-(j_{x}+1))f\left(\frac{j_{x}+1}{n}\right)\right\}}{\chi(nx-k)f\left(\frac{k}{n}\right)}>1,
\end{equation*} 
for all $x \in [c,d]$, $n \ge N_{2}$, and $k \in \{0,1,\dots,n\}$ such that $k>j_{x}+a_{n}$ (i.e., $k \notin I_{n,x}$). \newline Summarizing, if we set $N_{0}:=\max\{N_{1},N_{2}\}$, then
\begin{equation*}
\chi(nx-k)f\left(\frac{k}{n}\right)<\max\left\{\chi(nx-j_{x})f\left(\frac{j_{x}}{n}\right), \chi(nx-(j_{x}+1))f\left(\frac{j_{x}+1}{n}\right)\right\},
\end{equation*}
for every $x \in [c,d]$, $n \ge N_{0}$, and $k \in \{0,1,\dots,n\}$ such that $k < j_{x}-a_{n}$ or $k>j_{x}+a_{n}$. This completes the proof. 
\end{proof}
Now, we can prove a strong localization result for the (truncated) max-product generalized sampling operators based on strictly positive, even, centered bell-shaped kernel functions. 
\begin{teorema}
\label{Teorema5.1}
Let $\chi$ be a strictly positive kernel satisfying assumptions $(\chi1)$ and $(\chi3)-(\chi5)$. Furthermore, let $f$, $g:[0,1] \rightarrow \mathbb{R}^{+}$ be two bounded functions, both bounded below by strictly positive constants. Suppose there exist $a$, $b \in [0,1]$, with $0<a<b<1$, such that $f(x)=g(x)$ for all $x \in [a,b]$. Then, for any $c$, $d \in [a,b]$ with $a<c<d<b$, there exists $\overline{n} \in \mathbb{N}$, depending only on $a$, $b$, $c$, $d$, $f$, $g$, and $\chi$, such that 
\begin{equation*}
S_{n}^{\chi}(f)(x)=S_{n}^{\chi}(g)(x),
\end{equation*}
for all $x \in [c,d]$ and all $n \ge \overline{n}$. 
\end{teorema}
\begin{proof}
Let us fix $c,d \in [a,b]$, with $a<c<d<b$. By applying Lemma \ref{lemma5.1} to the function $f$, we obtain that there exists $N_{0} \in \mathbb{N}$ such that
\begin{equation*}
\bigvee_{k=0}^{n}\chi(nx-k) f\left(\frac{k}{n}\right)=\bigvee_{k \in I_{n,x}}\chi(nx-k)f\left(\frac{k}{n}\right),
\end{equation*}
for any $x \in [c,d]$, and $n \ge N_{0}$, where the set $I_{n,x}$ is defined as in the previous lemma. Recall that $N_{0}$ does not depend on $x$. \newline Now, let $x \in [c,d]$ and let $n \in \mathbb{N}$ with $n \ge N_{0}$ be fixed. Suppose there exists $k \in I_{n,x}$ (i.e., $j_{x}-\sqrt{n} \le k \le j_{x}+\sqrt{n}$, where $j_{x} \in \{0,1,\dots,n-1\}$ is such that $x \in \left[\frac{j_{x}}{n},\frac{j_{x}}{n}+\frac{1}{n}\right]$) with $\frac{k}{n} \notin [c,d]$. Then either $\frac{k}{n}<c$ or $\frac{k}{n}>d$. In the first case, we can write:
\begin{equation*}
0<c-\frac{k}{n} \le x-\frac{k}{n}  \le \frac{j_{x}+1}{n}-\frac{k}{n} \le \frac{\sqrt{n}+1}{n}. 
\end{equation*}
Since $\lim_{n\rightarrow+\infty} \frac{\sqrt{n}+1}{n}=0$, for sufficiently large $n$ we have $\frac{\sqrt{n}+1}{n}<c-a$, which implies that $\frac{k}{n} \in [a,c]$. Similarly, if $\frac{k}{n} >d$, then we have:
\begin{equation*}
0<\frac{k}{n}-d \le \frac{k}{n}-x \le \frac{k}{n}-\frac{j_{x}}{n} \le \frac{1}{\sqrt{n}}.
\end{equation*}
Observing that $\frac{1}{\sqrt{n}}\rightarrow 0 $ as $n \rightarrow +\infty$, for sufficiently large $n$ we have $\frac{1}{\sqrt{n}} < b-d$, and hence $\frac{k}{n} \in [d,b]$. Therefore, we conclude that there exists $\overline{N_{1}} \in \mathbb{N}$, independent of $x \in [c,d]$, but depending only on $a$, $b$, $c$, $d$, $f$, and $\chi$, such that
\begin{equation*}
\bigvee_{k=0}^{n}\chi(nx-k) f\left(\frac{k}{n}\right)= \bigvee_{k \in I_{n,x}} \chi(nx-k)f\left(\frac{k}{n}\right), \ \ n \ge \overline{N_{1}},
\end{equation*} 
and, moreover, $\frac{k}{n} \in [a,b]$ for all $k \in I_{n,x}$. Consequently, we can write:
\begin{equation*}
S_{n}^{\chi}(f)(x)=\frac{\displaystyle \bigvee_{k \in I_{n,x}}f\left(\frac{k}{n}\right)\chi(nx-k)}{\displaystyle \bigvee_{k=0}^{n}\chi(nx-k)},
\end{equation*}
with $\frac{k}{n} \in [a,b]$ for each $k \in I_{n,x}$, for all $x \in [c,d]$ and $n \in \mathbb{N}$ with $n \ge \overline{N_{1}}$. \newline Now, proceeding with the same arguments for the function $g$ as in the case of $f$, it follows that there exists $\overline{N_{2}} \in \mathbb{N}$, depending only on $a$, $b$, $c$, $d$, $g$ and $\chi$, such that
\begin{equation*}
S_{n}^{\chi}(g)(x)=\frac{\displaystyle \bigvee_{k \in I_{n,x}} g\left(\frac{k}{n}\right)\chi(nx-k)}{\displaystyle \bigvee_{k=0}^{n}\chi(nx-k)},
\end{equation*}
with $\frac{k}{n} \in [a,b]$ for each $k \in I_{n,x}$, for all $x \in [c,d]$ and $n \in \mathbb{N}$ with $n \ge \overline{N_{2}}$. Setting $\overline{n}:=\max\{\overline{N_{1}},\overline{N_{2}}\} \in \mathbb{N}$, and noting that, by assumption, $f\left(\frac{k}{n}\right)=g\left(\frac{k}{n}\right)$ for all $k \in I_{n,x}$, for any $x \in [c,d]$ and $n \ge \overline{n}$, we obtain the desired conclusion. 
\end{proof}
As a direct consequence of the above strong localization result, we can deduce a local shape- preserving property of the (truncated) max-product generalized sampling operators associated with strictly positive, even, centered bell-shaped kernel functions. More precisely, Corollary \ref{Corollario5.1} states that, for sufficiently large $n\in \mathbb{N}$, $S_{n}^{\chi}(f)$ can approximate strictly positive functions that are constant on subintervals of $[0,1]$ with a high degree of accuracy. Namely, if $f$ is a strictly positive function that is constant on some intervals $[a_{i},b_{i}]\subset[0,1]$, $i=1,\dots,p$, then, for sufficiently large $n$, $S_{n}^{\chi}(f)$ attains the same constant values on subintervals sufficiently close to each $[a_{i},b_{i}]$, $i=1,\dots,p$. Indeed, we can establish the following.
\begin{cor}
\label{Corollario5.1}
Under the same assumptions on $\chi$ as in the previous theorem, let $f:[0,1] \rightarrow \mathbb{R}^{+}$ be a bounded function with a strictly positive lower bound. Suppose, moreover, that there exist $a$, $b \in [0,1]$, with $0<a<b<1$, such that $f(x)=\alpha \in \mathbb{R}^{+}$ for every $x \in [a,b]$, i.e., $f$ is constant on $[a,b]$. Then, for any $c$, $d \in [a,b]$ with $a<c<d<b$, there exists $\overline{n} \in \mathbb{N}$, depending only on $a$, $b$, $c$, $d$, $f$ and $\chi$, such that $S_{n}^{\chi}(f)(x)=\alpha$ for all $x \in [c,d]$ and all $n \ge \overline{n}$.  
\end{cor}
\begin{proof}
Consider the function $g:[0,1] \rightarrow \mathbb{R}^{+}$, defined by $g(x)=\alpha>0$ for all $x \in [0,1]$. Since, by assumption, $f(x)=g(x)$, for every $x \in [a,b]$, and since it is easy to see that $S_{n}^{\chi}(g)(x)=\alpha$ for all $x \in [0,1]$, the conclusion follows directly from Theorem \ref{Teorema5.1}.  
\end{proof}
Note that Theorem \ref{Teorema5.1}, and hence the above local shape-preserving result, applies to the (truncated) max-product generalized sampling operators associated with all the centered bell-shaped functions listed at the end of Section \ref{Sezione2}, including the max-product neural network (NN) operators, with the exception of the central B-spline $\chi_{b}$ defined in (\ref{centralb-spline}) which is not strictly positive, being compactly supported.
\newline We point out that the above results show that max-product generalized sampling operators based on suitable kernels are particularly effective in the local reconstruction of non-smooth, strictly positive, bounded, and continuous functions. Indeed, for instance, if we consider a signal $f$ that vanishes on a subinterval $[c,d]$, we may assume that on this interval it attains an arbitrary small value $\varepsilon>0$. Hence, by Corollary \ref{Corollario5.1}, it follows that, for sufficiently large $n \in \mathbb{N}$, we have $S_{n}^{\chi}(f)(x)=\varepsilon$ for every $x \in [c,d]$. \newline These considerations suggest that the above theory may have potentially interesting applications in signal and image processing. In particular, the ability of the operators $S_{n}^{\chi}$ to locally approximate bounded, strictly positive functions with very high accuracy represents a clear advantage of the max-product generalized sampling operators over their classical (linear) counterparts, which do not exhibit this useful property. 
\newline We conclude with the following remark.
\begin{remark}
	Note that the statement of the strong localization result expressed in Theorem \ref{Teorema5.1}, as well as the consequent local shape-preserving property, requires the bounded function $f$ to be strictly positive on $[0,1]$. However, as in the case of the saturation theorem and the local inverse result, this restriction can be lifted. More precisely, let $f:[0,1]\rightarrow\mathbb{R}$ be a bounded function of variable sign on $[0,1]$, and choose a constant $c>0$ such that $f(x)+c>0$ for all $x\in [0,1]$. Then, it is sufficient to define the following max-product-type operator:
	\begin{equation*}
		L_{n}^{\chi}(f)(x):=S_{n}^{\chi}(f(\cdot)+c)(x)-c, \quad x \in [0,1], \ n \in \mathbb{N},
	\end{equation*}  
	For this operator, the following results hold:
	\begin{description}
		\item[{(i)}] Under the same assumptions on $\chi$ as in Theorem \ref{Teorema5.1}, let also $g:[0,1]\rightarrow\mathbb{R}$ be a bounded function on $[0,1]$ and suppose that there exist $a$, $b \in [0,1]$, with $0<a<b<1$, such that $f(x)=g(x)$ for all $x \in [a,b]$. Then, for all $c$, $d \in [0,1]$ satisfying $a<c<d<b$, there exists $\overline{n} \in \mathbb{N}$, depending only on $a$, $b$, $c$, $d$, $f$, $g$, and $\chi$, such that
		\begin{equation*}
			L_{n}^{\chi}(f)(x)=L_{n}^{\chi}(g)(x),
		\end{equation*}   
		for all $x\in [c,d]$ and for all $n \ge \overline{n}$. 
		\item[{(ii)}] Suppose that there exist $a$, $b \in [0,1]$, with $0<a<b<1$, such that $f(x)=\alpha\in \mathbb{R}$ for all $x \in [a,b]$, i.e., $f$ is constant on $[a,b]$. Then, for any $c$, $d \in [a,b]$ with $a<c<d<b$, there exists $\overline{n}\in \mathbb{N}$ depending only on $a$, $b$, $c$, $d$, $f$, and $\chi$, such that 
		\begin{equation*}
			L_{n}^{\chi}(f)(x)=\alpha,
		\end{equation*}   
		for all $x \in [c,d]$ and for all $n \ge \overline{n}$. 
	\end{description} 
\end{remark}

\section*{Acknowledgments}

The authors are members of the Gruppo Nazionale per l'Analisi Matematica, la Probabilit\`a e le loro Applicazioni (GNAMPA) of the Istituto Nazionale di Alta Matematica (INdAM), of the Gruppo UMI (Unione Matematica Italiana) T.A.A. (Teoria dell'Approssimazione e Applicazioni), and of the network RITA (Research ITalian network on Approximation).

\section*{Funding}
The authors have been supported within the project Fondazione Perugia: ``A.I.T.T.N. –
Algoritmi Innovativi di Imaging Digitale e Tecniche Termografiche per il Raffreddamento di Sistemi
Elettronici mediante Nanofluidi", 2026.

\section*{Conflict of interest/Competing interests}

{\small The authors declare that they have no conflict of interest and competing interest.}

\section*{Availability of data and material and Code availability}
{ \small Not applicable.}


\end{document}